\documentclass[12pt]{article}
\usepackage{amsmath,amsthm,amsfonts,latexsym,amsopn,verbatim,amscd,amssymb,}
\usepackage{amssymb}
\usepackage{amsfonts}
\usepackage{graphics,color}
\usepackage{graphicx}
\usepackage{amssymb,array,arydshln}
\usepackage{amsfonts}
\usepackage{amsbsy}
\usepackage{amssymb}
\usepackage{amsmath}
\usepackage{cite}
\usepackage{pst-all}
\usepackage{pstricks}
\usepackage{graphics}
\usepackage{float}

\theoremstyle{plain}
\newtheorem{Thm}{Theorem}[section]
\newtheorem{Lem}[Thm]{Lemma}

\newtheorem{Cor}[Thm]{Corollary}
\theoremstyle{definition}

\newtheorem{Def}[Thm]{Definition}

\numberwithin{equation}{section}

\begin{document}
\begin{center}
\textbf{Weakly nil clean graphs of rings }\\
\end{center}
\begin{center}
Ajay Sharma, Jayanta Bhattacharyya\\ and Dhiren Kumar Basnet$^1$ \\
\small{\it Department of Mathematical Sciences, Tezpur University,
 \\ Napaam, Tezpur-784028, Assam, India.\\
Email: ajay123@tezu.ernet.in,  jbhatta@tezu.ernet.in \\ and dbasnet@tezu.ernet.in}
\end{center}

\footnotetext[1]{corresponding author.}
\noindent \textit{\small{\textbf{Abstract:}  }}  In this article, the concept of nil clean graph of a ring has been generalised to weakly nil clean graph of a ring and graph theoretic properties like girth, clique number, diameter and chromatic index of the graph have been studied for a finite commutative ring.
\bigskip

\smallskip

\noindent \small{\textbf{\textit{$2010$ Mathematics Subject Classification:}}  05C75, 16N40, 16U99.} \\
\smallskip

\bigskip
\section{Introduction}

 Rings in this article are finite commutative rings with non zero identity. The set of idempotents, nilpotent elements of a ring $R$ and the $n\times n$ matrix ring over $R$ are denoted by $Idem(R)$, $Nil(R)$  and $M_n(R)$ respectively. An element $x$ of a ring $R$ is said to be a \textit{nil clean element}\cite{nc, ajd} if $x = n + e$, for some $e\in Idem(R)$ and $n\in Nil(R)$. The ring $R$ is said to be \textit{nil clean} if each element of $R$ is nil clean. An element $x$ of ring $ R$ is said to be a \textit{weakly nil clean element}\cite{cwnc , wnc} $x = n + e$ or $x=n-e$, for some $e\in Idem(R)$ and $n\in Nil(R)$. The ring $R$ is said to be \textit{weakly nil clean} if each element of $R$ is weakly nil clean. The set of nil clean and weakly nil clean elements of $R$ are denoted by $NC(R)$ and $WNC(R)$ respectively.

 In this article, by graph, we mean simple undirected graph. For a graph $G$, $V(G)$ denotes the set of vertices and $E(G)$ denotes the set of edges. For a positive integer $n$, P. Grimaldi\cite{gfr} defined and studied various properties of the unit graph $G(\mathbb{Z}_n)$, of the ring of integers modulo $n$, with vertex set $\mathbb{Z}_n$ and two distinct vertices $x$ and $y$ are adjacent if and only if $x+y$ is a unit. Further in \cite{ug}, authors generalized $G(\mathbb{Z}_n)$ to unit graph $G(R)$, where $R$ is an arbitrary associative ring with non zero identity. Recently D.K. Basnet and J. Bhattacharya \cite{ncg}, defined nil clean graph of a finite commutative ring $R$, where the vertex set is $R$ and two distinct vertices $x$ and $y$ are adjacent if and only if $x+y$ is nil clean and they studied various properties of nil clean graph of finite commutative ring.

In this article we have introduced weakly nil clean graph $G_{WN}(R)$ associated with a finite commutative ring $R$ motivated by the concept of \textit{Nil clean graph} of a finite ring $R$ \cite{ncg}. We define $G_{WN}(R)$ as $R$ is the vertex set and two vertices $x$ and $y$ are adjacent if and only if $x+y$ is a weakly nil clean element of $R$. The properties like girth, diameter, clique number and chromatic index of $G_{WN}(R)$ have been studied.\par
Here we mention some preliminaries about graph theory for this article. Let $G$ be a graph. The degree of the vertex $v\in G$ denoted by $deg(v)$, is the number of edges adjacent with $v$. A graph $G$ is said to be \textit{connected} if for any two distinct vertices of $G$, there is a path in $G$ connecting them. Number of edges on the shortest path between vertices $x$ and $y$ is called the \textit{distance} between $x$ and $y$ and is denoted by $d(x,y)$. If there is no path between $x$ and $y$ then we say $d(x,y)= \infty$. The diameter of a graph $G$, denoted by $diam(G)$, is the maximum of distances of each pair of distinct vertices in $G$. Also \textit{girth} of $G$ is the length of the shortest cycle in $G$, denoted by $gr(G)$. Note that if there is no cycle in $G$ then $gr(G)=\infty$. A complete graph is a simple undirected graph in which every pair of distinct vertices is connected by a unique edge. A bipartite graph $G$ is a graph whose vertices can be divided into two disjoint parts $V_1$ and $V_2$, such that $V(G) = V_1\cup V_2$ and every edge in $G$ has the form $e=(x,y) \in E(G)$, where $x\in V_1$ and $ y \in V_2$. Note that no two vertices both in $V_1$ or both in $V_2$ are adjacent. A complete bipartite graph is a graph where every vertex of the first part $V_1$ is connected to every vertex of the second part $V_2$, denoted by $K_{m,n}$, where $|V_1|=m$ and $|V_2|=n$. A complete bipartite graph $K_{1,n}$ is called star graph. \par
A clique is a subset of vertices of an undirected graph such that its induced subgraph is complete. A clique having $n$ number of vertices is called $n$-clique. The maximum clique of a graph is a clique such that there is no clique with more vertices. The clique number of a graph $G$ is denoted by $\omega (G)$ and defined by the number of vertices in the maximal clique of $G$.  An edge colouring of a graph $G$  is a map  $C: E(G) \rightarrow S$, where $S$ is a set of colours such that for all $e_1, e_2 \in E(G)$, if $e_1$ and $e_2$ are adjacent then $C(e_1) \neq C(e_2)$. The \textit{chromatic index} of a graph $G$ is denoted by $\chi^\prime(G)$ and is defined as the minimum number of colours needed for a proper colouring of $G$.

\section{Weakly nil clean graph}
In this section we will define weakly nil clean graph of a finite commutative ring and discuss its basic properties.
\begin{Def}
  The weakly nil clean graph of a ring $R$ denoted by $G_{WN}(R)$ is defined by setting $R$ as the vertex set and two distinct vertices $x$ and $y$ are adjacent if and only if $x+y$ is weakly nil clean. Here we are not considering any loop at a vertex in the graph.
\end{Def}
From the definition, nil clean graph of a ring $R$ is always a subgraph of a weakly nil clean graph of $R$.
\\ For illustration following are the weakly nil clean graphs of $\mathbb{Z}_{10}$ and $GF(25)$, where $\mathbb{Z}_{10}$ is the ring of integers modulo $10$ and $GF(25)$ is the finite field with $25$ elements.
\begin{figure}[H] 
\begin{pspicture}(0,4)(5,-1)
\scalebox{1}{
\rput(1,0){
\psdot[linewidth=.05](1,3)
\psdot[linewidth=.05](1,0)
\psdot[linewidth=.05](3,2)
\psdot[linewidth=.05](3,1)
\psdot[linewidth=.05](5,3)
\psdot[linewidth=.05](5,0)
\psdot[linewidth=.05](7,2)
\psdot[linewidth=.05](7,1)
\psdot[linewidth=.05](9,3)
\psdot[linewidth=.05](9,0)
\rput(1,3.3){$5$}
\rput(1,-.3){$0$}
\rput(3,2.3){$6$}
\rput(3,.7){$1$}
\rput(5,3.3){$4$}
\rput(5,-.3){$9$}
\rput(7,2.3){$7$}
\rput(7,.7){$2$}
\rput(9,3.3){$3$}
\rput(9,-.3){$8$}

\psline(1,0)(5,0)(3,2)(5,3)(1,3)(1,0)(3,1)(1,3)(3,2)(1,0)(5,3)(3,1)(9,0)(7,1)(9,3)(9,0)(7,2)(5,3)(7,1)(7,2)(9,0)(3,2)(9,3)
\psline(1,3)(5,0)(3,1)(9,3)(7,2)(5,0)(7,1)
}}
\end{pspicture}
\caption{Weakly nil clean graph of $\mathbb{Z}_{10}$.}
\end{figure}
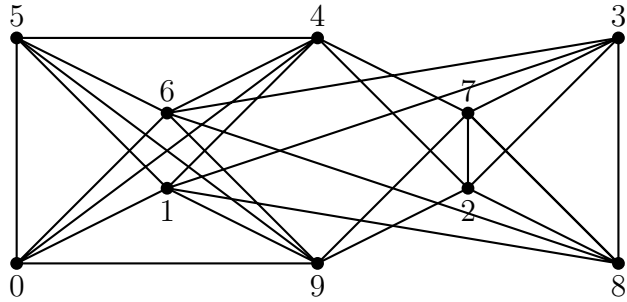
 Since, $ GF(25)\cong \mathbb{Z}_5[x]/<x^2+x+1>=\{ax+b+<x^2+x+1>\,\,:\,\,a,b\in \mathbb{Z}_5\}$.

Define, $\alpha=x+<x^2+x+1> $, then we have $GF(25)=\{0, 1, 2, 3, 4, \alpha, 2\alpha, 3\alpha, 4\alpha, 1+\alpha, 1+2\alpha, 1+3\alpha, 1+4\alpha, 2+\alpha, 2+2\alpha, 2+3\alpha, 2+4\alpha, 3+\alpha, 3+2\alpha, 3+3\alpha, 3+4\alpha, 4+\alpha, 4+2\alpha, 4+3\alpha, 4+4\alpha\}$, Observe that $WNC(GF(25))=\{0,1,4\}$.

\begin{figure}[H]  
\begin{pspicture}(0,3)(0,-7)
\scalebox{.8}{
\rput(5,0){
\psdot[linewidth=.05](2,3)
\psdot[linewidth=.05](0,2)
\psdot[linewidth=.05](6,3)
\psdot[linewidth=.05](2,1)
\psdot[linewidth=.05](6,1)
\psline(0,2)(2,3)(6,3)(6,1)(2,1)(0,2)
\psline(2,3)(2,1)
\rput(2,3.3){$1$}
\rput(0,2.3){$ 0$}
\rput(6,3.3){$3$}
\rput(2.3,1.3){$4$}
\rput(6.3,1.3){$2$}
\psdot[linewidth=.05](1,0)
\psdot[linewidth=.05](5,0)
\psdot[linewidth=.05](-1,-1)
\psdot[linewidth=.05](3,-1)
\psdot[linewidth=.05](7,-1)
\psdot[linewidth=.05](-1,-2)
\psdot[linewidth=.05](3,-2)
\psdot[linewidth=.05](7,-2)
\psdot[linewidth=.05](1,-3)
\psdot[linewidth=.05](5,-3)
\psline(-1,-1)(1,0)(3,-1)(5,0)(7,-1)(7,-2)(5,-3)(3,-2)(1,-3)(-1,-2)(-1,-1)
\psline(-1,-1)(3,-2)
\psline(-1,-2)(3,-1)
\psline(1,0)(7,-1)
\psline(1,-3)(7,-2)
\psline(5,0)(5,-3)
\rput(-1.3,-1){$\alpha$}
\rput(-1.3,-2){$4\alpha$}
\rput(1,0.3){$4\alpha +1$}
\rput(1,-3.3){$\alpha+1$}
\rput(3,-0.7){$\alpha +4$}
\rput(3,-2.3){$4\alpha +4$}
\rput(5,0.3){$4\alpha +2$}
\rput(5,-3.3){$\alpha +2$}
\rput(8,-1){$\alpha+3$}
\rput(8,-2){$4\alpha+3$}

\psdot[linewidth=.05](1,-4)
\psdot[linewidth=.05](5,-4)
\psdot[linewidth=.05](-1,-5)
\psdot[linewidth=.05](3,-5)
\psdot[linewidth=.05](7,-5)
\psdot[linewidth=.05](-1,-6)
\psdot[linewidth=.05](3,-6)
\psdot[linewidth=.05](7,-6)
\psdot[linewidth=.05](1,-7)
\psdot[linewidth=.05](5,-7)
\psline(-1,-5)(1,-4)(3,-5)(5,-4)(7,-5)(7,-6)(5,-7)(3,-6)(1,-7)(-1,-6)(-1,-5)
\psline(-1,-5)(3,-6)
\psline(-1,-6)(3,-5)
\psline(1,-4)(7,-5)
\psline(1,-7)(7,-6)
\psline(5,-4)(5,-7)
\rput(-1.3,-5){$2\alpha$}
\rput(-1.3,-6){$3\alpha$}
\rput(1,-3.7){$3\alpha +1$}
\rput(1,-7.3){$2\alpha+1$}
\rput(3,-4.7){$2\alpha +4$}
\rput(3,-6.3){$3\alpha +4$}
\rput(5,-3.7){$3\alpha +2$}
\rput(5,-7.3){$2\alpha +2$}
\rput(8,-5){$2\alpha+3$}
\rput(8,-6){$3\alpha+3$}
}}
\end{pspicture}
\caption{Nil clean graph of $GF(25)$}\label{gf25}
\end{figure}
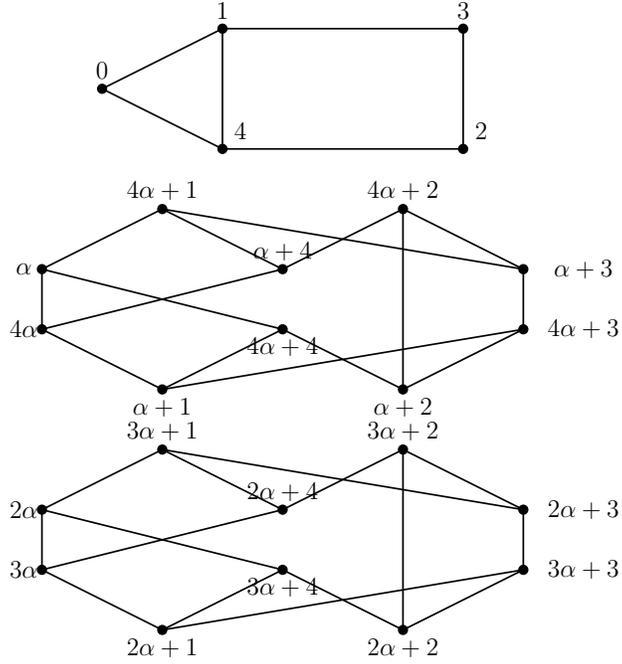
For any prime $p$ the weakly nil clean graph of $\mathbb{Z}_p$ is given below. If $\frac{p-1}{2}\equiv 0 (mod\,\,2)$ then the graph $G_{WN}(\mathbb{Z}_p)$ is

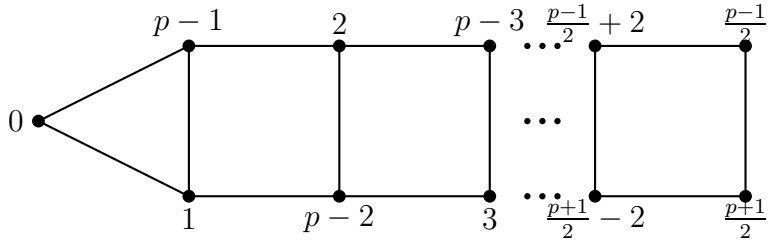
\begin{figure}[H]  
\begin{pspicture}(0,4)(0,0)
\scalebox{1}{
\rput(2,0){
\psdot[linewidth=.05](1,3)
\psdot[linewidth=.05](3,3)
\psdot[linewidth=.05](5,3)
\psdot[linewidth=.007](5.5,3)
\psdot[linewidth=.007](5.7,3)
\psdot[linewidth=.007](5.9,3)
\psdot[linewidth=.05](6.4,3)
\psdot[linewidth=.05](8.4,3)
\psdot[linewidth=.05](1,1)
\psdot[linewidth=.05](3,1)
\psdot[linewidth=.05](5,1)
\psdot[linewidth=.007](5.5,1)
\psdot[linewidth=.007](5.7,1)
\psdot[linewidth=.007](5.9,1)
\psdot[linewidth=.05](6.4,1)
\psdot[linewidth=.05](8.4,1)
\psdot[linewidth=.05](-1,2)
\psdot[linewidth=.007](5.5,2)
\psdot[linewidth=.007](5.7,2)
\psdot[linewidth=.007](5.9,2)
\psline(1,3)(5,3)(5,1)
\psline(1,3)(1,1)(5,1)
\psline(1,1)(-1,2)(1,3)
\psline(3,3)(3,1)
\psline(6.4,3)(8.4,3)(8.4,1)(6.4,1)(6.4,3)

\rput(-1.3,2){$ 0$}

\rput(1,3.3){$p-1$}
\rput(1,.7){$1$}
\rput(3,.7){$p-2$}
\rput(3,3.3){$2$}
\rput(5,0.7){$3$}
\rput(5,3.3){$p-3$}
\rput(6.4,3.3){$\frac{p-1}{2}+2$}
\rput(8.4,3.3){$\frac{p-1}{2}$}
\rput(6.4,0.7){$\frac{p+1}{2}-2$}
\rput(8.4,0.7){$\frac{p+1}{2}$}

}}
\end{pspicture}
\caption{Weakly nil clean graph of $\mathbb{Z}_p$, if $\frac{p-1}{2}\equiv 0 (mod\,\,2)$}\label{Zp1}
\end{figure}
If $\frac{p-1}{2}\equiv 1 (mod\,\,2)$, then the graph $G_{WN}(\mathbb{Z}_p)$ is

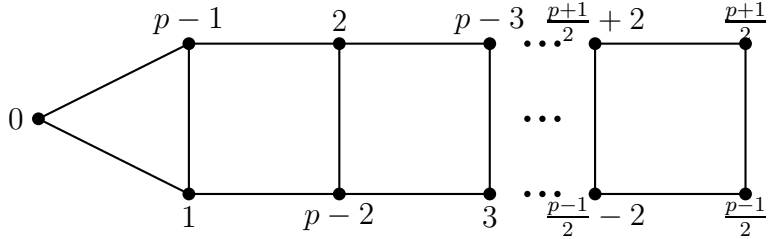
\begin{figure}[H]  
\begin{pspicture}(0,4)(0,0)
\scalebox{1}{
\rput(2,0){
\psdot[linewidth=.05](1,3)
\psdot[linewidth=.05](3,3)
\psdot[linewidth=.05](5,3)
\psdot[linewidth=.007](5.5,3)
\psdot[linewidth=.007](5.7,3)
\psdot[linewidth=.007](5.9,3)
\psdot[linewidth=.05](6.4,3)
\psdot[linewidth=.05](8.4,3)
\psdot[linewidth=.05](1,1)
\psdot[linewidth=.05](3,1)
\psdot[linewidth=.05](5,1)
\psdot[linewidth=.007](5.5,1)
\psdot[linewidth=.007](5.7,1)
\psdot[linewidth=.007](5.9,1)
\psdot[linewidth=.05](6.4,1)
\psdot[linewidth=.05](8.4,1)
\psdot[linewidth=.05](-1,2)
\psdot[linewidth=.007](5.5,2)
\psdot[linewidth=.007](5.7,2)
\psdot[linewidth=.007](5.9,2)
\psline(1,3)(5,3)(5,1)
\psline(1,3)(1,1)(5,1)
\psline(1,1)(-1,2)(1,3)
\psline(3,3)(3,1)
\psline(6.4,3)(8.4,3)(8.4,1)(6.4,1)(6.4,3)

\rput(-1.3,2){$ 0$}

\rput(1,3.3){$p-1$}
\rput(1,.7){$1$}
\rput(3,.7){$p-2$}
\rput(3,3.3){$2$}
\rput(5,0.7){$3$}
\rput(5,3.3){$p-3$}
\rput(6.4,3.3){$\frac{p+1}{2}+2$}
\rput(8.4,3.3){$\frac{p+1}{2}$}
\rput(6.4,0.7){$\frac{p-1}{2}-2$}
\rput(8.4,0.7){$\frac{p-1}{2}$}

}}
\end{pspicture}
\caption{Weakly nil clean graph of $\mathbb{Z}_p$, if $\frac{p-1}{2}\equiv 1 (mod\,\,2)$ }\label{Zp2}
\end{figure}
The following theorem is based on complete graph.
\begin{Thm}
  The weakly nil clean graph $G_{WN}(R)$ is a complete graph if and only if $R$ is a weakly nil clean ring.
\end{Thm}
\begin{proof}
  Let $G_{WN}(R)$ be a complete graph. For $r\in R$, $r$ is adjacent to $0$, so $r=r+0$ is weakly nil clean element in $R$. Hence $R$ is weakly nil clean ring. Converse is obvious from the definition of weakly nil clean graph.
\end{proof}
\begin{Lem}
  Let $R$ be a ring. If $x+Nil(R)$ and $y+Nil(R)$ are adjacent in $G_{WN}(R/Nil(R))$ then every element of $x+Nil(R)$ is adjacent to each element of $y+Nil(R)$ in the weakly nil clean graph $G_{WN}(R)$.
\end{Lem}
\begin{proof}
  Let $x+Nil(R)$ and $y+Nil(R)$ are adjacent in $G_{WN}(R/Nil(R))$. Then $(x+y)+Nil(R)=e+Nil(R)$ or $(x+y)+Nil(R)=-e+Nil(R)$, where $e\in Idem(R)$, as idempotents lift modulo $Nil(R)$. Thus $x+y=e+n$ or $x+y=-e+n$, where $n\in Nil(R)$. Now for $a\in x+Nil(R)$ and $b\in y+Nil(R)$, we have $a=x+n_1$ and $b=y+n_2$, for some $n_1,n_2\in Nil(R)$. So $a+b=e+(n-n_1-n_2)$ or $a+b=-e+(n-n_1-n_2)$. Hence $a$ and $b$ are adjacent in $G_{WN}(R)$.
\end{proof}
The next Lemma is related to the degree of a vertex in $G_{WN}(R)$.
\begin{Lem}\label{deg}
Let $G_{WN}(R)$ be the weakly nil clean graph of a ring $R$. For $x \in R$ we have the following:
\begin{enumerate}
  \item If $2x$ is weakly nil clean, then $deg(x) = |WNC(R)| - 1$.
 \item If $2x$ is not weakly nil clean, then $deg(x) = |WNC(R)|$.
\end{enumerate}
\end{Lem}
\begin{proof}
  Proof is similar to Lemma 2.4 \cite{ncg}.
\end{proof}
The result of connectedness of weakly nil clean graph of $\mathbb{Z}_n$ and $M_n(\mathbb{Z}_n)$ is given below:
\begin{Thm}
 For a ring $R$. The following hold :
 \begin{enumerate}
   \item $G_{WN}(R)$ need not be connected.
   \item Let $R = \mathbb{Z}_n$. For $a \in \mathbb{Z}_n$ there is a path from $a$ to $0$.
   \item $G_{WN}(\mathbb{Z}_n)$ is connected.
   \item Let $R = \mathbb{Z}_n$. For $A \in M_n(\mathbb{Z}_n)$ there is a path from $A$ to $0$, where $0$ is the zero matrix of $M_n(\mathbb{Z}_n)$.
   \item $G_{WN}(M_n(\mathbb{Z}_n))$ is connected.
 \end{enumerate}
 \end{Thm}
 \begin{proof}
   $(i)$ It is clear from the graph $G_{WN}(G(25))$, Figure $2$. $(ii),\,\, (iii),\,\,(iv)$ and $(v)$ are clear from Theorem 2.5 \cite{ncg}, as nil clean graph of a ring $R$ is a subgraph of a weakly nil clean graph of $R$.
   \end{proof}

 \section{Invariants of weakly nil clean graph}
 In this section, for any ring $R$, we study about some invariants like girth, clique number, diameter and chromatic index of $G_{WN}(R)$.
\subsection{Girth of $G_{WN}(R)$}
Here in the given theorem we show that the girth of weakly nil clean graph of any ring is $3$.
\begin{Thm}
 For a ring $R$ with $|R|\geq 3$, girth of a weakly nil clean graph $G_{WN}(R)$ is 3.
\end{Thm}
\begin{proof}
 For a ring $R$, clearly $\{0,1,-1\}\subseteq WNC(R)$, so the weakly nil graph of $R$ contains the given subgraph
 \begin{figure}[H]
\begin{pspicture}(0,-.5)(0,2.2)
\centering
\scalebox{1}{
\rput(-2,0){
\psdot[linewidth=.05](7,1)
\psdot[linewidth=.05](9,2)
\psdot[linewidth=.05](9,0)
\psline(7,1)(9,2)(9,0)(7,1)
\rput(6.7,1){$0$}
\rput(9.7,2){$-1$}
\rput(9.4,0){$1$}
}}
\end{pspicture}
\caption{A cycle of length $3$ in $G_{WN}(R)$.}
\end{figure}
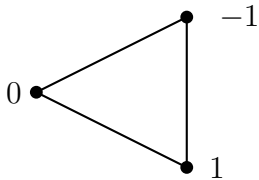
So the girth of a ring $R$ is always $3$.
\end{proof}
\begin{Cor}
  The weakly nil clean graph of a ring $R$ can not be bipartite.
\end{Cor}
\begin{Cor}
  Weakly nil clean graph of any ring $R$ can not be a star graph.
\end{Cor}
\subsection{Clique number of $G_{WN}(R)$}
Here we study about the clique number of weakly nil clean graph of any finite field and graph $G_{WN}(\mathbb{Z}_{2p})$.
\begin{Thm}\label{C1}
  The clique number of the field $\mathbb{Z}_p$ is $3$, where $p$ is prime  and  $p \geq 3$.
\end{Thm}
\begin{proof}If $p=3$ then clearly $\omega( G_{WN}(\mathbb{Z}_3))=3$. Let $p > 3$.
  The weakly nil clean graph of $\mathbb{Z}_p$ is connected, as nil clean graph of $\mathbb{Z}_n$ is connected and girth of $G_{WN}(\mathbb{Z}_p)$ is $3$. If possible, assume that $\omega( G_{WN}(\mathbb{Z}_p))>3$. Then there exist four vertices say $\{v_1, v_2, v_3, v_4\}$ having $deg(v_i)\geq 3$ and the subgraph induced by $\{v_1, v_2, v_3, v_4\}$ is complete. Since  $G_{WN}(\mathbb{Z}_p)$ is connected and $p\geq 5$, so there exists $v_j\in \{v_1, v_2, v_3, v_4\}$ such that $deg(v_j)>3$, which is a contradiction as $WNC(\mathbb{Z}_p)=\{0, 1, -1\}$.
\end{proof}
It is clear from Figure $5$ \cite{ncg},  that the weakly nil clean graph of $GF(p^n)$ is partitioned into disjoint subgraphs, where one subgraph contains $p$ number of vertices and rest other subgraph contains each of  $2p$ number of vertices.
\begin{Thm}
  The clique number of weakly nil clean graph of any finite field is $3$.
\end{Thm}
\begin{proof}
Let, $GF(p^n)$ be a field of order $p^n$. Let $\{G_1, G_2, G_3, \cdot \cdot \cdot, G_m\}$ be the disjoint connected subgraphs of $G_{WN}(GF(p^n))$, where $G_1$ contains $p$ number of vertex and $G_i$ contains $2p$ number of vertex, for $2\leq i\leq m$. Clearly
 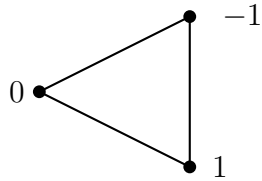
\begin{figure}[H]
\begin{pspicture}(0,-.5)(0,2.2)
\centering
\scalebox{1}{
\rput(-2,0){
\psdot[linewidth=.05](7,1)
\psdot[linewidth=.05](9,2)
\psdot[linewidth=.05](9,0)
\psline(7,1)(9,2)(9,0)(7,1)
\rput(6.7,1){$0$}
\rput(9.7,2){$-1$}
\rput(9.4,0){$1$}
}}
\end{pspicture}
\caption{A cycle of length $3$ in $G_{WN}(GF(p^n))$.}
\end{figure}
is a subgraph of $G_{WN}(GF(p^n))$. If possible let, clique number of $G_{WN}(GF(p^n))$ is 4. Then there exist a $4$-$clique$ which lies in $G_i$, for some $i$ with $1\leq i\leq m$. By similar procedure of Theorem \ref{C1}, we get a contradiction, as required.
\end{proof}
In graph theory, the neighbourhood of a vertex $v$ of a graph $G$ is denoted by $N(v)$ and defined to be the set of all vertices adjacent to $v$. Note that $v\notin N(v)$.
\begin{Lem}\label{L2p}
  For $\mathbb{Z}_{2p}$, where $p\geq 5$ is a prime number, the following hold:
  \begin{enumerate}
    \item $N(a)\cap N(-a)=\phi$, for $a\in \mathbb{Z}_{2p}$ and $\,\,a,-a\notin \{0,1, p, p+1, \frac{p+1}{2}, \frac{p-1}{2}\}$.
    \item $N(a)\cap N(b)=\phi$, for $a,b\in \mathbb{Z}_{2p}$ with $a+b=1$ and $a\notin \{0, 1, p, p+1, \frac{p-1}{2}, \frac{p+1}{2}, \frac{p+3}{2}, \frac{3p-1}{2}, \frac{3p+1}{2}, \frac{3p+3}{2}\}$.
    \item $N(a)\cap N(b)=\phi$, for $a,b\in \mathbb{Z}_{2p}$ with $a+b=-1$ and $a\notin \{0, -1, p, p-1, \frac{p-1}{2}, \frac{p+1}{2}, \frac{3p-3}{2}, \frac{p-3}{2}, \frac{3p-1}{2}, \frac{3p+1}{2}\}$.
  \end{enumerate}
\end{Lem}
\begin{proof}
  $(i)$ If $a\notin \{0,1, p+1, p, \frac{p+1}{2}, \frac{p-1}{2}\}$, then the possible values of $a$ are  the elements of $S=\{2,3,\cdot\cdot\cdot, \frac{p-3}{2}, p-2, p-3, \cdot\cdot\cdot, \frac{p+3}{2} \}$ and from Figure $7$, it is clear that $N(a)\cap N(-a)=\phi$, for $a\in S$.\\
  $(ii)$ The possible pairs $(a,b)\in\mathbb{Z}_{2p}\times \mathbb{Z}_{2p}$ satisfying the condition $a+b=1$ and $a\notin \{0, 1, p, p+1, \frac{p-1}{2}, \frac{p+1}{2}, \frac{p+3}{2}, \frac{3p-1}{2}, \frac{3p+1}{2}, \frac{3p+3}{2}\}$ are $S=\{(-1,2), (-2,3), \cdot\cdot\cdot, (\frac{3p+5}{2},\frac{p-3}{2})\}\cup \{(p-1,p+2), (p-2, p+3), \cdot\cdot\cdot, (\frac{p+5}{2}, \frac{3p-3}{2})\}$. For $(a,b)\in S$, from Figure $7$ of weakly nil clean graph of $\mathbb{Z}_{2p}$, we conclude that $N(a)\cap N(b)=\phi$. Similarly $(iii)$ follows from Figure $7$.

\end{proof}
\begin{Thm}\label{TC2p}
  There are exactly five $4$-cliques of $G_{WN}(\mathbb{Z}_{2p})$, where $p\geq 5$ is a prime number.
\end{Thm}
\begin{proof}
 It is clear that $WNC(\mathbb{Z}_{2p})=\{0, 1, -1, p, p-1, p+1 \}$. Let $C=\{a_1,a_2,a_3,a_4\}$ be a $4$-clique of $G_{WN}(\mathbb{Z}_{2p})$. So $a_i+a_j\in WNC(\mathbb{Z}_{2p})$, for any $1\leq i,j\leq4$. Without loss of generality we consider $a_1$ and $a_2$.\\
Case I: If $a_1+a_2=0$ then $a_1=-a_2$. If $0\in C$ then $a_i\in WNC(\mathbb{Z}_{2p})$, for $1\leq i\leq4$ and it gives only two $4$-cliques $viz.$ $\{1,-1,p,0 \}$ and $\{p-1, p+1, p,0\}$.
If $0\notin C$ and $a_1\in WNC(\mathbb{Z}_{2p})$ then $a_1=1$ or $p+1$. If $a_1=1$ then $a_2=-1$ and $a_3=p$ as $0\notin C$, this must imply $a_4=0$, which is a contradiction. Similarly if $a_1=p+1$ then also we get a contradiction. If $0\notin C$ and $a_1\notin WNC(\mathbb{Z}_{2p})$ then from Lemma \ref{L2p} $(i)$, $a_1=\frac{p+1}{2}$ or $\frac{p-1}{2}$. If $a_1=\frac{p+1}{2}$ then $a_2=\frac{3p-1}{2}$ and it gives the $4$-clique $\{\frac{p+1}{2}, \frac{p-1}{2}, \frac{3p+1}{2}, \frac{3p-1}{2}\}$. Similarly if $a_2=\frac{p-1}{2}$ then also it gives the same $4$-clique $\{\frac{p+1}{2}, \frac{p-1}{2}, \frac{3p+1}{2}, \frac{3p-1}{2}\}$.\\
Case II: Let $a_1+a_2=1$. Here $a_i+a_j\in \{-1,p,p-1,p+1\}$ except $a_1+a_2$, where $1\leq i, j \leq 4$. Here two cases arise
\begin{enumerate}
  \item If $a_i\in WNC(\mathbb{Z}_{2p})$ then the possible pair for $(a_1, a_2)$ are $(0,1)$ and $(p,p+1)$. If $a_1=0$ and $a_2=1$ then it gives $a_3=\{-1,p,p-1\}$. Since $p+(p-1)$ and $(-1)+p\in WNC(\mathbb{Z}_{2p})$, so we get the $4$-cliques $\{0,1,-1,p\}$ and $\{0,1,p,p-1\}$. Similarly if $a_1=p$ and $a_2=p+1$ then it gives two $4$-cliques $\{p,p+1,p-1,0\}$ and $\{p,p+1,-1,0\}$.
  \item Let $a_i\notin WNC(\mathbb{Z}_{2p})$, for some $1\leq i \leq 4$. Without loss of generality assume that $i=1$. By Lemma \ref{L2p}, $a_1\in \{\frac{p-1}{2}, \frac{p+1}{2}, \frac{3p-1}{2}\}$. If $a_1=\frac{p+1}{2}$ then it gives the $4$-clique namely $\{\frac{p-1}{2}, \frac{p+1}{2}, \frac{3p-1}{2}, \frac{3p+1}{2}\}$. Again if $a_1=\frac{p-1}{2}$ or $\frac{3p-1}{2}$ then we will get a contradiction that $\{a_1, a_2, a_3, a_4\}$ is a $4$-clique.
\end{enumerate}
Case III: If $a_1+a_2=-1$, then similar to Case II, we conclude that the total possibilities of $4$-cliques are $\{0,-1,1,p\}$, $\{p,p-1,1,0\}$, $\{p,p-1,p+1,0\}$, $\{0,-1,p,p+1\}$ and $\{\frac{p-1}{2}, \frac{3p-1}{2}, \frac{p+1}{2}, \frac{3p+1}{2}\}$.\\
Case IV: Let $a_i+a_j\in \{p, p-1,p+1\}$. Assume $a_1$ is fixed then the other three vertices of the $4$-clique are $p-a_1,\,\,p+1-a_1$ and $p-1-a_1$, so $(p-a_1)+(p+1-a_1)=2p-2a_1+1\in \{p,p-1,p+1\}$. If $2p-2a_1+1=p$ then $a_1=\frac{p+1}{2}$ and $p+1-a_1=\frac{p+1}{2}$, a contradiction. Also $2p-2a_1+1\neq p-1$ and $2p-2a_1+1\neq p+1$. Hence in this case no such $4$-clique exists.\par

  From the above $4$ cases we conclude that the only $4$-cliques of the weakly nil clean graph of $\mathbb{Z}_{2p}$ are $\{0,1,-1,p\},\{0,1,p,p-1\},\{0,-1,p,p+1\},\{0,p,p-1,p+1\}$ and $\{\frac{p-1}{2},\frac{p+1}{2},\frac{3p-1}{2}, \frac{3p+1}{2}\}$

\end{proof}
\begin{Thm}
  The clique number of $G_{WN}(\mathbb{Z}_{2p})$ is $4$.
\end{Thm}
\begin{proof}
  If possible let $\{a,b,c,d,e\}$ be a $5$-clique. Then clearly it contains any one of the $4$-cliques mentioned in Theorem \ref{TC2p}.\\
  Case I: Suppose $\{a,0, 1, -1,p\}$ be a $5$-clique. Then $a+0\in WNC(\mathbb{Z}_{2p})=\{0, 1, -1, p, p-1, p+1 \}$. If $a+0=0,1,-1$ or $p$ then $a=0,1,-1$ or$p$, a contradiction. If $a+0=p+1$ then $a=p+1$ implies $(p+1)+1\in WNC(\mathbb{Z}_{2p})$, a contradiction. Similarly if $a+0=p-1$ then we get $p-2\in WNC(R)$, a contradiction. Hence there is no possibility of $5$-clique containing the $4$-clique $\{0, 1, -1,p\}$.\\
  Case II: If the $5$-clique contains $4$-clique $\{0,1,p,p-1\}$ or $\{0,p,p+1,-1\}$ or $\{p,p+1, p-1,0\}$, then similar as Case I, we get a contradiction.\\
  Case III: Suppose $\{a, \frac{p-1}{2}, \frac{p+1}{2}, \frac{3p-1}{2}, \frac{3p+1}{2}\}$ be a $5$-clique. Then $a+\frac{p+1}{2}\in WNC(\mathbb{Z}_{2p})$. If $a+\frac{p+1}{2}=0$ then $a=\frac{3p-1}{2}$, a contradiction. If $a+\frac{p+1}{2}=1$ then $a=\frac{3p+1}{2}$, a contradiction. If $a+\frac{p+1}{2}=-1$ then $a=\frac{3p-3}{2}$ implies $\frac{3p-3}{2} + \frac{3p-1}{2}=p-2\in WNC(\mathbb{Z}_{2p})$, a contradiction. If $a+\frac{p+1}{2}=p$ then $a=\frac{p-1}{2}$, a contradiction. If $a+\frac{p+1}{2}=p+1$ then $a=\frac{p+1}{2}$, a contradiction and if $a+\frac{p+1}{2}=p-1$ then $a=\frac{p-3}{2}$ implies $\frac{p-3}{2}+\frac{p-1}{2}=p-2\in WNC(\mathbb{Z}_{2p})$, a contradiction.\par
  From the above cases and Theorem \ref{TC2p} we conclude that $5$-clique does not exist in $G_{WN}(\mathbb{Z}_{2p})$. Hence the clique number of $G_{WN}(\mathbb{Z}_{2p})$ is $4$.
\end{proof}

\subsection{Diameter}
\begin{Lem}\label{D1}
  $R$ is weakly nil clean ring if and only if $diam(G_{WN}(R))=1$.
\end{Lem}
Here we define an element $x\in R$ weakly nil clean element of type $1$ if $x=n+e$, and is of type $2$ if $x=n-e$, where $n\in Nil(R)$ and $e\in Idem(R)$.
\begin{Thm}
  Let $R$ and $S$ are two weakly nil clean rings but not nil clean rings then $diam(G_{WN}(R\times S))$ is either $2$ or $3$.
\end{Thm}
\begin{proof}
  Clearly $R\times S$ is not weakly nil clean ring by Theorem 2.3 \cite{wnc} and $diam(R\times S)\geq 2$ by Lemma \ref{D1}. Let $(x_1,x_2),(y_1,y_2)\in R\times S$, so $x_i=n_i+ e_i$ or $x_i=n_i- e_i$ and $y_i=m_i+ f_i$ or $y_i=m_i-f_i$, where $n_i\in Nil(R)$, $e_i\in Idem(R)$, $m_i\in Nil(S)$, $f_i\in Idem(S)$ and $1\leq i\leq2$. \par
  Case I: If all $x_i$ and $y_i$ are weakly nil clean element of same type then $(x_1,x_2)\leftrightarrow (0,0)\leftrightarrow (y_1,y_2)$ is a path from $(x_1,x_2)$ to $(y_1,y_2)$ in $G_{WN}(R\times S)$. \par
  Case II: If any two from $x_i$ and $y_i$ are weakly nil clean elements of different types then without loss of generality assume $x_1=n_1-e_1$ and all are weakly nil clean elements of type 1. Clearly $(x_1,x_2)\leftrightarrow (e_1,0)\leftrightarrow (0,0)\leftrightarrow (y_1,y_2)$ is a path from $(x_1,x_2)$ to $(y_1,y_2)$.\par
  Case III: If $x_1$ and $x_2$ are weakly nil clean elements of different types and $y_1$ and $y_2$ are also weakly nil clean elements of different types. Assume that $x_1$ and $y_2$ are of type 1 and $x_2$ and $y_1$ are of type 2, then $(x_1,x_2)\leftrightarrow (0,e_2)\leftrightarrow (f_1,0)\leftrightarrow (y_1,y_2)$ is a path from $(x_1,x_2)$ to $(y_1,y_2)$. Similarly for any pair satisfying Case III we get a path of length $3$ from  $(x_1,x_2)$ to $(y_1,y_2)$.\par
  Hence from the above three cases we conclude that the diameter of $G_{WN}(R\times S)$ is either $2$ or $3$.

\end{proof}
The weakly nil clean graph of $\mathbb{Z}_{2p}$ is given below for the next theorem on diameter of $G_{WN}(\mathbb{Z}_{2p})$.
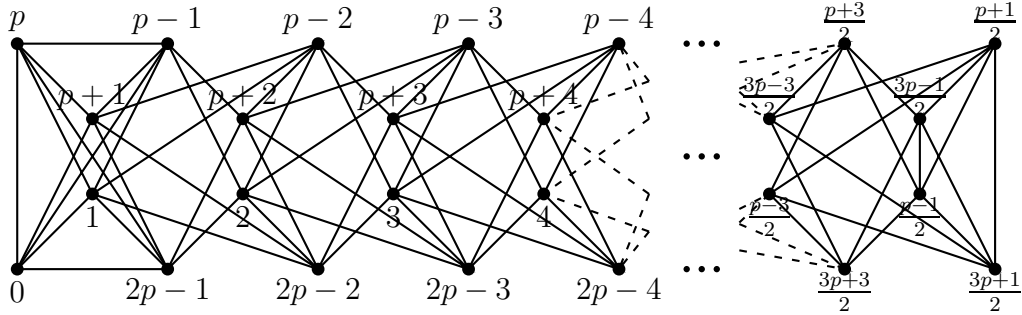
\begin{figure}[H] 
\begin{pspicture}(0,4)(5,-1)
\scalebox{1}{
\rput(1,0){
\psdot[linewidth=.05](-1,3)
\psdot[linewidth=.05](-1,0)
\psdot[linewidth=.05](0,2)
\psdot[linewidth=.05](0,1)
\psdot[linewidth=.05](1,3)
\psdot[linewidth=.05](1,0)
\psdot[linewidth=.05](2,2)
\psdot[linewidth=.05](2,1)
\psdot[linewidth=.05](3,3)
\psdot[linewidth=.05](3,0)
\psdot[linewidth=.05](4,2)
\psdot[linewidth=.05](4,1)
\psdot[linewidth=.05](5,3)
\psdot[linewidth=.05](5,0)
\psdot[linewidth=.05](6,2)
\psdot[linewidth=.05](6,1)
\psdot[linewidth=.05](7,3)
\psdot[linewidth=.05](7,0)
\rput(-1,3.3){$p$}
\rput(-1,-.3){$0$}
\rput(0,2.3){$p+1$}
\rput(0,.7){$1$}
\rput(1,3.3){$p-1$}
\rput(1,-.3){$2p-1$}
\rput(2,2.3){$p+2$}
\rput(2,.7){$2$}
\rput(3,3.3){$p-2$}
\rput(3,-.3){$2p-2$}
\rput(4,2.3){$p+3$}
\rput(4,.7){$3$}
\rput(5,3.3){$p-3$}
\rput(5,-.3){$2p-3$}
\rput(6,2.3){$p+4$}
\rput(6,.7){$4$}
\rput(7,3.3){$p-4$}
\rput(7,-.3){$2p-4$}

\psline(-1,3)(-1,0)(1,0)(0,1)(-1,0)(0,2)(-1,3)(1,3)(0,1)(3,0)
\psline(-1,0)(1,3)(0,2)(3,3)(0,1)(-1,3)(1,0)(2,1)(1,3)(2,2)(1,0)
\psline(1,0)(0,2)(3,0)(2,1)(3,3)(2,2)(3,0)
\psline(5,3)(2,2)(5,0)(2,1)(5,3)(4,2)(5,0)(4,1)(5,3)(6,2)(5,0)(6,1)(5,3)
\psline(3,3)(4,2)(3,0)(4,1)(3,3)
\psline(7,3)(4,2)(7,0)(4,1)(7,3)(6,2)(7,0)(6,1)(7,3)
\psline[linestyle=dashed,dash=3pt 4pt](7.4,1)(6,2)(7.4,2.5)
\psline[linestyle=dashed,dash=3pt 4pt](7.4,2)(6,1)(7.4,.5)
\psline[linestyle=dashed,dash=3pt 4pt](7.4,2)(7,3)(7.4,2.5)
\psline[linestyle=dashed,dash=3pt 4pt](7.4,1)(7,0)(7.4,.5)
\psdot[linewidth=.007](7.9,3)
\psdot[linewidth=.007](8.1,3)
\psdot[linewidth=.007](8.3,3)
\psdot[linewidth=.007](7.9,1.5)
\psdot[linewidth=.007](8.1,1.5)
\psdot[linewidth=.007](8.3,1.5)
\psdot[linewidth=.05](9,2)
\psdot[linewidth=.05](9,1)
\psdot[linewidth=.05](10,3)
\psdot[linewidth=.05](10,0)
\psdot[linewidth=.05](11,2)
\psdot[linewidth=.05](11,1)
\psdot[linewidth=.05](12,3)
\psdot[linewidth=.05](12,0)
\rput(12,3.3){$\frac{p+1}{2}$}
\rput(12,-.3){$\frac{3p+1}{2}$}
\rput(11,2.3){$\frac{3p-1}{2}$}
\rput(11,.7){$\frac{p-1}{2}$}
\rput(10,3.3){$\frac{p+3}{2}$}
\rput(10,-.3){$\frac{3p+3}{2}$}
\rput(9,2.3){$\frac{3p-3}{2}$}
\rput(9,.7){$\frac{p-3}{2}$}
\psline(12,0)(12,3)(11,1)(12,0)(11,2)(12,3)(9,2)(12,0)(9,1)(12,3)
\psline(10,3)(11,2)(10,0)(11,1)(10,3)(9,2)(10,0)(9,1)(10,3)
\psline[linestyle=dashed,dash=3pt 4pt](8.6,2.75)(10,3)(8.6,2.4)
\psline[linestyle=dashed,dash=3pt 4pt](8.6,2.36)(9,2)(8.6,2.2)
\psline[linestyle=dashed,dash=3pt 4pt](8.6,.25)(10,0)(8.6,.6)
\psline[linestyle=dashed,dash=3pt 4pt](8.6,.8)(9,1)(8.6,.62)
\psline(11,2)(11,1)
\psdot[linewidth=.007](7.9,0)
\psdot[linewidth=.007](8.1,0)
\psdot[linewidth=.007](8.3,0)

}}
\end{pspicture}
\caption{Weakly nil clean graph of $\mathbb{Z}_{2p}$.}\label{Z_2p}
\end{figure}
\begin{Thm}
  The following hold for the ring $\mathbb{Z}_n:$
  \begin{enumerate}
    \item If $n=2^k3^l$, for some integer $k\geq 0$ and $l\geq 0$, then $diam(G_{WN}(\mathbb{Z}_n))=1$.

    \item For a prime $p$, $diam(G_{WN}(\mathbb{Z}_p))=\frac{p-1}{2}$.
    \item If $n=2p$, where $p\geq 5$ is an odd prime, then $diam(G_{WN}(\mathbb{Z}_n))=\frac{p-1}{2}$.
    \item For a finite field $F$ having order $p^k$, $diam(G_{WN}(F))=\infty$ if and only if $k>1$.
  \end{enumerate}
\end{Thm}
\begin{proof}
  $(i)$ is follows from Theorem $2.9$ \cite{wnc} and Lemma $3.9$. $(ii)$ and $(iii)$ can be shown easily from the weakly nil clean graph of $\mathbb{Z}_p$ in Figure $3$, $4$ and weakly nil clean graph of $\mathbb{Z}_{2p}$ in Figure $7$. $(iv)$ follows from Figure $5$ \cite{ncg}.
\end{proof}
\subsection{Chromatic index}
Let $\vartriangle$ be the maximum vertex degree of $G$, then Vizing’s theorem \cite{gt} gives $\vartriangle \leq \chi^\prime(G) \leq \vartriangle+ 1$. Vizing’s theorem divides the graphs into two classes according to their chromatic index; graphs satisfying $\chi^\prime(G) = \vartriangle$ are called graphs of class $1$, those with $\chi^\prime(G) =\vartriangle + 1$ are graphs of class $2$.

Following theorem shows that $G_{WN}(R)$ is of class $1$.

\begin{Thm}
Let $R$ be a ring then the weakly nil clean graph of $R$ is of class $1$.
\end{Thm}

\begin{proof} We colour the edge $ab$ by the colour $a + b$. By this colouring, every two distinct edges $ab$ and $ac$ have different colours and $C = \{ a + b\,\,|\,\, ab\mbox{ is an edge in }G_{WN}(R) \}$ is the set of colours. Therefore the weakly nil clean graph has a $|C|$-edge colouring and so $\chi^\prime(G_{WN}(R)) \leq |C|$. But $C \subset WNC(R)$ and $\chi^\prime(G_{WN}(R)) \leq |C| \leq |WNC(R)|$. By Lemma \ref{deg},  $\vartriangle \leq |WNC(R)|$ and so by Vizing's theorem, we have $\chi^\prime(G_{WN}(R))\geq \vartriangle = |WNC(R)|$. Therefore $\chi^\prime(G_{WN}(R)) = |WNC(R)| =\vartriangle,$ i.e., $G_{WN}(R)$ is of class $1$.
\end{proof}

\end{document}